\documentclass[11pt]{amsart}
\usepackage{amsmath,amsfonts,amscd,amssymb,amsthm,epsfig,euscript}
\usepackage{mathrsfs}
\usepackage{amsxtra}
\usepackage[all]{xy}
\usepackage{verbatim,epsf}
\newtheorem{theorem}{Theorem}
\newtheorem{proposition}{Proposition}
\newtheorem{lemma}{Lemma}
\newtheorem{corollary}{Corollary}
\theoremstyle{definition}
\newtheorem{definition}{Definition}

\theoremstyle{remark}
\newtheorem{remark}{Remark}
\numberwithin{equation}{section}
\newcommand{\field}[1]{\ensuremath{\mathbb{#1}}}
\newcommand{\CC}{\field{C}}
\newcommand{\DD}{\field{D}}
\newcommand{\HH}{\field{H}}
\newcommand{\PP}{\field{P}}

\newcommand{\ZZ}{\field{Z}}

\DeclareMathOperator{\im}{Im}
\DeclareMathOperator{\re}{Re}

\newcommand{\delb}{\bar\partial}

\newcommand{\fS}{\mathfrak{S}}

\newcommand{\curly}[1]{\mathscr{#1}}

\newcommand{\cC}{\curly{C}}

\newcommand{\cF}{\curly{F}}

\newcommand{\cH}{\curly{H}}
\newcommand{\cI}{\curly{I}}

\newcommand{\cK}{\curly{K}}

\newcommand{\cN}{\curly{N}}

\newcommand{\cP}{\curly{P}}

\newcommand{\cT}{\curly{T}}
\newcommand{\cU}{\curly{U}}
\newcommand{\cV}{\curly{V}}

\newcommand{\ga}{\gamma}
\newcommand{\s}{\gamma}

\newcommand{\Ga}{\Gamma}

\newcommand{\PSL}{\mathrm{PSL}(2,\mathbb{R})}

\newcommand{\Ur}{\mathrm{U}(r)}

\newcommand{\NN}{\field{N}}

\DeclareMathOperator{\tr}{tr} 
 
\DeclareMathOperator{\Aut}{Aut}
 
\DeclareMathOperator{\GL}{GL} \DeclareMathOperator{\End}{End}
 \DeclareMathOperator{\Ad}{Ad}

\usepackage{hyperref}
\hypersetup{colorlinks,citecolor=blue,plainpages=false,hypertexnames=false}
\begin{document}

\title[Poincar\' e series]{On vector-valued Poincar\' e series of weight 2}


\author[Meneses]{Claudio Meneses}
\address{Department of Mathematics\\
Stony Brook University\\ Stony Brook, NY 11794, USA\\
Current address:\\ Centro de Investigaci\'on en Matem\'aticas, A.C.\\
   Jalisco S/N, Col. Valenciana CP: 36023 Guanajuato, Gto.\\
   M\'exico \\}
\email{claudio.meneses@cimat.mx} 
\thanks{}

\thanks{}


\date{}

\dedicatory{}

\begin{abstract}

Given a pair $(\Ga,\rho)$ of a Fuchsian group of the first kind, and a unitary representation $\rho$ of $\Ga$ of arbitrary rank, the problem of construction of vector-valued Poincar\'e series of weight 2 is considered. Implications in the theory of parabolic bundles are discussed. When the genus of the group is zero, it is shown how an explicit basis for the space of these functions can be constructed. 

\end{abstract}

\maketitle

\tableofcontents

\section{Introduction}

For a finite dimensional vector space $V$ over $\CC$ with a Hermitian inner product and $\rho:\Gamma\to\Aut V$ a unitary representation of a Fuchsian group $\Ga$, assumed irreducible without loss of generality,
\emph{an automorphic form of weight 2 with the representation} $\rho$ is a holomorphic function $f:\HH\to V$ satisfying 
\begin{equation}\label{automorphic}
f(\gamma\tau)\gamma'(\tau)=\rho(\gamma) f(\tau),\qquad \forall\gamma\in\Ga,\;\tau\in\HH.
\end{equation}
The case $V=\CC$ is due to Petersson \cite{Petersson40,Petersson48}. We will assume that either $V=\CC^{r}$ or $V=\End \CC^{r}$, with their standard inner products.

Our focus is to emphasize the geometric nature of automorphic forms. The significance of weight 2 comes from the fact that for representations of the form $\Ad\circ\rho$, matrix-valued cusp forms of weight 2 can be used to model infinitesimal deformations of geometric structures, corresponding 
to stable parabolic bundles on a compact Riemann surface \cite{N-S65,MS80, TZ07}.\footnote{In principle, the correspondence can be extended to an arbitrary compact Lie group $G$ with a given complex representation in $V$ (e.g. its complexified adjoint representation), for which a suitable version of Petersson inner product is also defined, provided that one considers the corresponding algebro-geometric notion of a \emph{parahoric torsor}, introduced by Balaji--Seshadri in \cite{BS15}, which generalizes the notion of parabolic bundles when $G = \mathrm{U}(r)$. For simplicity, we will only consider the unitary case.} It is thus desirable to establish a bridge between these two now classical subjects. For instance, the correspondence may be used to generate the spaces of classical cusp forms of arbitrary positive even weight  \cite{Men16}.  At a deeper level, the author's original motivation to develop the present work, is the geometric significance of the Petersson inner product on the spaces of matrix-valued cusp forms of weight 2 with a representation of the form $\Ad\circ\rho$, in the sense that it provides an analog of the Weil--Petersson K\"ahler metric on the corresponding moduli spaces of stable parabolic bundles, which was originally introduced and studied in a more differential--geometric setup in the works of Narasimhan and Atiyah--Bott \cite{Nar70,AB83}.

The subject of vector-valued automorphic forms has been studied already by several authors (e.g., the work of Knopp and Mason \cite{KM03,KM04} for the modular group and sufficiently large weight, and \cite{SS14}).\footnote{I would like to thank the referee for pointing out the recent preprint \cite{CF17}.} However, a general study of the weight 2 case for arbitrary Fuchsian groups seems to be lacking.
The present investigation undertakes Lehner's dictum \cite{Lehner64} as a guiding motif, which lists the three primary problems in the classical theory of automorphic forms as: (1) proving their existence, (2) constructing their Poincar\' e series, and (3) finding a minimal spanning family among these.  Our approach is twofold, and mixes aspects of the theory of holomorphic vector bundles over compact Riemann surfaces with the complex analysis on the upper half-plane.\footnote{Incidentally, the restriction to unitary representations could be lifted, provided that the conjugacy classes over elliptic and parabolic generators are kept fixed. This is the case since there is an open set $\cK^{0}_{\CC}$ in a corresponding complex $\mathrm{GL}(r,\CC)$-character variety $\cK_{\CC}$, consisting of equivalence classes of irreducible representations $\chi:\Gamma\to \mathrm{GL}(r,\CC)$, inducing stable parabolic structures (see remark \ref{even-split}), and which is isomorphic, under the Mehta--Seshadri correspondence, to the holomorphic cotangent bundle $T^{*}\cN$ to a moduli space of stable parabolic bundles $\cN$ (cf. \cite{Men15}). In this sense, for such generic $[\chi]\in\cK^{0}_{\CC}$, to every space of $\chi$-automorphic forms of weight 2, there corresponds a unique space of $\rho$-automorphic forms of weight 2, with $\rho$ unitary irreducible, up to equivalence in $\cK^{0}_{\CC}$. The correspondence is, however, highly nontrivial, and will not be considered here.}

The geometric interpretation of cusp forms as global holomorphic sections of a vector bundle enables the computation of the dimension of the space they span by means of the Riemann-Roch theorem (a foundational result implicit in the work of Andr\'e Weil \cite{Weil38}), thus solving (1) (cf. \cite{Hejhal83}). The main difficulty with (2) is that the standard Poincar\'e series fail to be absolutely convergent. This can be remedied  following Hecke's idea  \cite{Hecke27}, later extended by Petersson in the rank 1 case \cite{Petersson48}, on the introduction of a convergence factor. The present approach is a simplification and extension to arbitrary rank.  The basic principle that we implement is that vector-valued Poincar\'e series of weight 2 can be thought of as limits of certain absolutely convergent series \emph{within} the Hilbert space of measurable automorphic forms of weight 2. Thus, their construction follows from the properties of the Cauchy--Riemann operators on Hilbert spaces. We take the opportunity to remark that all the techniques developed in this work (the construction of Poincar\'e series, Petersson inner product properties, etc.) extend, and in fact greatly simplify, when one considers automorphic forms of arbitrary positive even weight. Keeping that in mind, we will only work with the weight 2 case. 

As our main result, we solve (3) in section \ref{genus 0} for the genus 0 case. An approach based on a  vector bundle analog of the uniformization theorem is possible, and is presented. Concretely, theorems \ref{theo-1} and \ref{theo-2} indicate how to choose a collection of independent Poincar\'e series spanning the space of cusp forms $\mathfrak{S}_{2}(\Ga,\rho)$ and $\mathfrak{S}_{2}(\Ga,\Ad\rho)$. We expect that the same geometric techniques may shed some light on the general case in arbitrary genus (which, to the best of our knowledge, is unsettled even in the classical scalar case), a problem we plan to return to, as well as to study further relations with the theory of moduli of parabolic bundles, in a separate publication.


The contents of this work are organized as follows: Section \ref{pre} is devoted to introducing conventions; in section \ref{geo} we recall how the dimension of the space of cusp forms of weight 2 can be determined by means  of the Riemann-Roch theorem on a suitable vector bundle; in section \ref{theo}, Poincar\' e series are constructed by means of a limiting procedure in the Hilbert space of measurable automorphic forms, and moreover, it is shown how, with the aid of the Petersson inner product, a complete family can be determined; in section \ref{genus 0} it is shown how a basis for the space of cusp forms can be prescribed in the genus zero case. 

\section{Preliminary remarks}\label{pre}

Let $\Ga$ be a Fuchsian group (i.e. a discrete subgroup of $\PSL$)  of the first kind.\footnote{In the classical 
literature, Fuchsian groups are also called \emph{principal-circle groups}, or simply \emph{circle groups}, 
and a Fuchsian group of the first kind is often called \emph{horocyclic}. $H$-groups are also called 
\emph{zonal}.} Among these we will be interested in the so-called \emph{H-groups} (\emph{Grenzkreisgruppen}). 
By definition, an $H$-group is a Fuchsian group of the first kind that is finitely generated and possesses 
parabolic elements. Examples of these model the fundamental groups of Riemann surfaces of finite type (compact 
surfaces with a finite number of points removed) but in general torsion is also allowed. It is a classical 
result \cite{Hejhal83,Lehner64}
that $\Gamma$ admits an explicit presentation in terms of $2g$ hyperbolic generators 
$A_{1},\dots,A_{g},B_{1},\dots,B_{g},$ $m$ elliptic generators $S_{1},\dots,S_{m}$ and $n$ parabolic generators 
$T_{1},\cdots,T_{n}$ with relations 
\[
\prod_{i=1}^{g}[A_{i},B_{i}]\cdot\prod_{j=1}^{m}S_{j}\cdot\prod_{k=1}^{n}T_{k}=e,  
\]
\[
S_{1}^{\nu_{1}}=\cdots =S_{m}^{\nu_{m}}=e.
\]
The \emph{signature} of $\Gamma$ is the collection of nonnegative integers $(g;\nu_{1},\dots,\nu_{m};n)$ with $2\leq \nu_{j}<\infty$, $j=1,\cdots,m$, subject to the inequality $$2g-2 +\sum_{i=1}^{m}\left(1-\frac{1}{\nu_{i}}\right)+n > 0.$$ The equivalence relation on Fuchsian groups is conjugation in $\PSL$.

$\Gamma$ admits a fundamental region $\mathfrak{F}$ bounded by a finite number of geodesic arcs and vertices determined by the fixed points of $\{S_{i},T_{j}\}$. We will denote by $\{e_{1},\cdots,e_{m}\}\subset\mathbb{H}$ the set of elliptic fixed points ($S_{i}(e_{i})=e_{i}$), and by $\{p_{1},\cdots,p_{n}\}\subset\mathbb{R}\cup\{\infty\}$ the set of cusps ($T_{i}(p_{i})=p_{i}$) in the fundamental region $\mathfrak{F}$. We follow Shimura's convention \cite{Shimura71} on the construction of the compact quotient
$S=\Gamma\setminus\mathbb{H}^{*}$ with local elliptic (resp. parabolic) coordinates  $\zeta_{i}=\zeta^{\nu_{i}}\circ\sigma^{-1}_{i}$ (resp. $q_{i}=q\circ\sigma_{i}^{-1}$) where $\zeta:\HH\to\DD$ is the Cayley mapping, $q=e^{2\pi\sqrt{-1}\tau}$, and $\sigma_{i}\in \text{PSL}(2,\mathbb{R})$ satisfy 
\begin{equation}\label{can-elliptic}
\zeta\left((\sigma_{i}^{-1}S_{i}\sigma_{i})(\tau)\right)
=e^{2\pi\sqrt{-1}/\nu_{i}}\cdot\zeta(\tau),\qquad 1\leq i \leq m
\end{equation}
\begin{equation}\label{can-cusp}
\left(\sigma_{i}^{-1}T_{i-m}\sigma_{i}\right)(\tau)=\tau\pm 1, \qquad m+1 \leq i \leq m+n.
\end{equation}

The spectral decomposition of a unitary matrix $M\in \Ur$ can be equivalently understood in terms of the data $\{W,[U]\}$ consisting of a diagonal matrix of \emph{weights} 
\[
W = 
\begin{pmatrix}
\alpha_{1} & & 0\\
& \ddots &\\
0 & & \alpha_{r}
\end{pmatrix}, 
\qquad 0\leq\alpha_{1}\leq\dots\leq\alpha_{r}<1
\]
and a \emph{partial flag} 
\[
[U]\in\cF_{W} = \mathrm{T}_{W}\setminus\Ur,
\]
where $\mathrm{T}_{W}=\mathrm{Z}\left(e^{2\pi\sqrt{-1}W}\right)$, so that 
\[
M=U e^{2\pi\sqrt{-1}W} U^{-1}.
\]
Thus, for a unitary representation $\rho:\Ga\rightarrow \Ur$, the collections of matrices $\{\rho(S_{i})\}_{i=1}^{m}$, $\{\rho(T_{i})\}_{i=1}^{n}$ determine collections $\{W_{i},[U_{i}]\}_{i=1}^{m+n}$ of elliptic and parabolic data
(cf. \cite{MS80}). The diagonal elements $\alpha_{ij}$ of the matrices $W_{i}$ are called \emph{elliptic} ($1\leq i \leq m$) and \emph{parabolic} ($m+1 \leq i \leq m+n$) \emph{weights}.\footnote{The notation we follow here mimics that of the theory of parabolic bundles.} 

\begin{remark}\label{flag manifold}
The homogeneous space $\cF_{W}$ admits a unique complex structure invariant under $\GL(r,\CC)$, obtained in terms of the isomorphism 
\[
\mathrm{T}_{W}\setminus\Ur\cong \mathrm{P}_{W}\setminus\GL(r,\CC),
\]  
where $\mathrm{P}_{W}$ is the parabolic subgroup of $\GL(r,\CC)$ with Lie algebra $\mathfrak{z}\left( W \right)+\mathfrak{n}(r)$, and $\mathfrak{n}(r)$ is the Lie algebra of nilpotent lower triangular matrices.
\end{remark}
\begin{remark}
Notice that for $1\leq i \leq m$, $\alpha_{ij}=n_{ij}/\nu_{i}$ for $n_{ij}\in\NN$, $0\leq n_{ij}<\nu_{i}$. This imposes a strong restriction on the possible dimensions of the partial flags $\cF_{W_{i}}$ for elliptic generators, depending on the values of $r$ and $\nu_{i}$. 
\end{remark}

Recall that an automorphic form of weight 2 is called \emph{holomorphic} or \emph{regular}, if at each cusp $p_{1},\dots,p_{n}$, 
\[
\displaystyle\lim_{\im(\tau)\to\infty}f(\sigma_{i}\tau)\sigma_{i}'(\tau) \quad\mathrm{exists,}
\]  
and a \emph{cusp form} if moreover 
\[
\lim_{\im(\tau)\to\infty}f(\sigma_{i}\tau)\sigma_{i}'(\tau)=0,\quad i=m+1,\cdots,m+n.
\]
In other words, $\phi=f(\tau)d\tau$ is a vector-valued holomorphic differential on $\mathbb{H}$ satisfying $\gamma^{*}(\phi)=\rho(\gamma)\phi$ $\forall\gamma\in\Gamma$, and with specific boundary behavior. 

\begin{remark}
The requirement $m+n>0$ shall be emphasized, although the geometric interpretation of automorphic forms of weight 2 that we will present in the next section is still valid without this restriction, in a rather trivial way: in the case of purely hyperbolic Fuchsian groups, the quotient $\Gamma\setminus \HH$ is compact, and there is an isomorphism 
\[
\Gamma \cong \pi_{1}(\Gamma\setminus\HH, [\tau_{0}])
\]
for some arbitrary auxiliary point $[\tau_{0}]\in \Gamma\setminus\HH$. Moreover, the notions of  holomorphic (regular) automorphic forms and cusp forms are vacuous, and hence one can establish an isomorphism, in the spirit of proposition \ref{iso-reg-cusp}, between the space of automorphic forms of weight 2 for the pair $(\Gamma, \rho)$, and the vector space
\[
H^{0}\left(\Gamma\setminus\HH, E_{\rho}\otimes K_{\Gamma\setminus \HH}\right)
\] 
where $E_{\rho}$ denotes the local system over $\Gamma\setminus\HH$ associated to $\rho$.
\end{remark}


\begin{remark}
For any given choice of unitary matrix representatives $\{U_{i}\}_{i=1}^{n+m}$ of the parabolic and elliptic data, any regular vector-valued automorphic form of weight 2 admits a  power series expansion at each elliptic fixed point ($1\leq i\leq m$)
\begin{equation}\label{elliptic}
\frac{f(\sigma_{i}\tau)\sigma_{i}'(\tau)}{(\zeta'(\tau)/\zeta(\tau))}
=U_{i}\zeta^{\nu_{i}W_{i}}\left(\sum_{k=0}^{\infty}\mathbf{b}_{i}(k)\zeta^{k\nu_{i}}\right),
\end{equation}
and a $q$-series expansion near each cusp ($m+1 \leq i \leq m+n$)
\begin{equation}\label{Fourier}
f(\sigma_{i}\tau)\sigma_{i}'(\tau)=U_{i}q^{W_{i}}\left(\sum_{k=0}^{\infty}
\mathbf{b}_{i}(k)q^{k}\right),
\end{equation}
where $\mathbf{b}_{i}(k) = {}^{t} \left(b_{i}(k)_{1},\dots, b_{i}(k)_{r}\right)$.
The holomorphicity condition for $f$ at an elliptic fixed point is then that $b_{i}(0)_{j}=0$ whenever the weight $\alpha_{i j}$ vanishes. It is important to emphasize that the previous series expansions \eqref{elliptic}--\eqref{Fourier} are suited to the choice of elliptic and parabolic weights for $\rho$ in the Mehta--Seshadri convention for parabolic stability \cite{MS80}. We had chosen to follow such convention as we want to exploit the correspondence between the theory of vector-valued automorphic forms and the moduli theory of parabolic bundles.
\end{remark} 

\begin{remark}\label{vanishing-coef}
The vanishing order of a cusp form differs from that of a regular form only  when a given parabolic weight vanishes, $\alpha_{ij}=0$, since then necessarily $b_{i}(0)_{j}=0$.
In particular, if a representation $\rho$ has only nonzero parabolic and elliptic weights, the spaces of regular and cusp forms coincide.
\end{remark}

\begin{remark}
It is customary to consider the divisor of an automorphic form to have fractional coefficients at elliptic fixed points, emphasizing the orbifold nature of the surface $S=\Ga\setminus\HH^{*}$ \cite{Lehner64,Shimura71}. In the right-hand side expansion of \eqref{elliptic}, we emphasize instead the geometric meaning of $f$ as  a holomorphic differential in a choice of a local coordinate for $S$.
\end{remark}

\begin{remark}\label{q-series}
In the case of representations of type $\Ad\rho=\Ad\circ\rho\cong \rho^{\vee}\otimes\rho$, where $\Ad$ denotes the adjoint representation of $\GL(r,\CC)$, it follows that for every element $\Phi\in\mathfrak{S}_{2}(\Ga,\Ad\rho)$, the expansions \eqref{elliptic} take the special form
\begin{equation}\label{elliptic1}
U_{i}\zeta^{W_{i}}_{i}\left(\sum_{k=0}^{\infty}B_{i}(k)\zeta_{i}^{k}\right)\zeta^{-W_{i}}_{i}U_{i}^{-1}.
\end{equation}
while the Fourier series expansions  \eqref{Fourier} are
\begin{equation}\label{Fourier1}
U_{i}q^{W_{i}}\left(\sum_{k=0}^{\infty}B_{i}(k)q^{k}\right)q^{-W_{i}}U_{i}^{-1},
\end{equation}
Together, the cusp form condition, and the holomorphicity at elliptic fixed points, are then equivalent to $B_{i}(0)\in\mathfrak{n}(W_i)$, $i=1,\dots,m+n$, where $\mathfrak{n}(W_{i})$ is the ideal of $\mathfrak{n}(r)$ complementary to $\mathfrak{z}(W_{i})\cap \mathfrak{n}(r)$, or equivalently, the $jk$-entry of $B_{i}(0)$ is zero whenever $\alpha_{ij}-\alpha_{ik}\geq 0$. The Lie algebra $\mathfrak{n}(W_{i})$ models the tangent space at every point of the partial flag manifold $\cF_{W_{i}}$.
\end{remark}

We will denote the spaces of regular automorphic forms and cusp forms of weight 2 by $\mathfrak{M}_{2}(\Gamma,\rho)$ and $\mathfrak{S}_{2}(\Gamma,\rho)$ respectively.  

\section{Geometric interpretation}\label{geo}

We will recall here the rather classical construction of a vector bundle $E_{\rho}\rightarrow S$ with prescribed system of trivializations, 
explicitly depending on the choice of parabolic and elliptic weights of $\rho$ (cf. \cite{Rohrl57,N-S65,MS80}). 
On $\HH_{0}=\mathbb{H}\setminus\Ga\cdot\{e_{1},\dots,e_{m}\}$ we consider the local system $E_{0} = \HH_{0}\times\mathbb{C}^{r}/\sim$ over $X=\Ga\setminus\HH_{0}$,
where $(\tau,v)\sim (\gamma\tau,\rho(\gamma)v)$, for all $\gamma\in\Gamma$. To extend $\text{pr}:E_{0}\rightarrow X$ to the elliptic fixed points and cusps, make 
\[
\cU_{i}=\left\{
\begin{array}{lr}
\langle S_{i}\rangle\setminus\DD_{\epsilon}(e_{i}), &  1 \leq i \leq m,\\\\
\langle T_{i}\rangle\setminus((\sigma_{i}\cdot\HH_{\delta})\cup \{p_{i}\}), & m+1 \leq i \leq m+n,
\end{array}\right.
\]
where $\HH_{\delta}=\{\tau\in\HH\,|\,\im{\tau}>\delta\}$, $\epsilon=e^{-2\pi\delta}$ and $\delta>0$ is sufficiently large so that the collection of open neighborhoods $\{\mathscr{U}_{i}\}_{i=1}^{n+m}$ is pairwise disjoint in $S$. For $i=1,\cdots,m$, $j=1,\cdots,n$, we define the functions
\[
\Psi_{i}:\text{pr}^{-1}(\mathscr{U}_{i}\setminus\{e_{j}\})\rightarrow\cU_{i}\times\mathbb{C}^{r}, \qquad
\Phi_{j}:\text{pr}^{-1}(\mathscr{U}_{m+j}\setminus\{p_{j}\})\rightarrow\cU_{m+j}\times\mathbb{C}^{r},
\]
\begin{equation}
\Psi_{i}([\tau,v])=\left([\tau],\zeta_{i}^{-W_{i}}U_{i}^{-1}v\right),\qquad
\Phi_{j}([\tau,v])=\left([\tau],q_{j}^{-W_{m+j}}U_{m+j}^{-1}v\right),
\end{equation}
which depend on a choice of unitary matrix representatives $\{U_{i}\}_{i=1}^{n+m}$ of the parabolic and elliptic data and are readily seen to define biholomorphisms onto their images. The bundle $E_{\rho}$ is defined as the identification
\[
E_{\rho}:=E_{0}\sqcup_{\Psi_{i}}(\mathbb{D}_{\epsilon}\times\mathbb{C}^{r})
\sqcup_{\Phi_{j}}(\mathbb{D}_{\epsilon}\times\mathbb{C}^{r}).
\]
We can consider a sufficiently fine covering of $S$ by extending $\{\mathscr{U}_{i}\}_{i=1}^{m+n}$ with a finite collection of domains  $\{\mathscr{V}_{j}\}$ covering $X\subset S$, thus obtaining 

\begin{lemma}\label{trans}
For any choice of unitary matrix representatives $\{U_{i}\}_{i=1}^{m+n}$ of the parabolic and elliptic data, the vector bundle $E_{\rho}$ is determined by the transition functions 
\begin{equation}\label{transition-ext}
g_{ij}=\left\{
\begin{array}{ll}
\rho(\gamma_{ij}) & \mbox{on $\cV_{ij}$},\\\\
\zeta_{i}^{-W_{i}}U_{i}^{-1} & \mbox{on $\mathscr{U}_{i}\cap\mathscr{V}_{j}$, $i \leq m$},\\\\
q_{i}^{-W_{i}}U_{i}^{-1} & \mbox{on $\mathscr{U}_{i}\cap\mathscr{V}_{j}$, $i>m$}. 
\end{array}\right. 
\end{equation}
\end{lemma}
\begin{proof}
This is a simple extension of the proof considered in \cite{N-S65}, remark 6.2. It is clear that any two choices of unitary matrix representatives would yield equivalent systems of transition functions.
\end{proof}
\begin{remark}\label{degree-sections}
The bundle $E_{\rho}$ satisfies (cf. \cite{Weil38,MS80})
$$c_{1}(E_{\rho})=-\displaystyle\sum_{i=1}^{n+m}\sum_{j=1}^{r}\alpha_{ij},
\quad\quad H^{0}(S,E_{\rho})\cong(\CC^{r})^{\rho}.$$
When parabolic or elliptic generators are present, it is not longer true that $E^{\vee}_{\rho}=E_{\overline{\rho}}$, where $\overline{\rho}= \rho^{\vee}=\;^{t}\rho^{-1}$ is the contragradient (or dual) representation. Moreover, it is neither true, in general, that $E_{\Ad\rho}\cong \End E_{\rho}$, and consequently that $E^{\vee}_{\Ad\rho}\cong E_{\Ad\rho}$ (see proposition \ref{iso-reg-cusp} and remark \ref{nonzero}). Observe how, over the parabolic bundles, the structures induced by parabolic and elliptic generators are, up to the finite order condition, indistinguishable.
\end{remark}
\begin{lemma}\label{dual}
The parabolic and elliptic weights of $\overline{\rho}$ are related to those of  $\rho$ as
\begin{equation}\label{permutation-weights}
\alpha '_{ij}=\left\{
\begin{array}{cl}
0 &\text{if}\;\;j\leq s_{i},\\
1-\alpha_{ir+s_{i}+1-j} & \text{if}\;\; j>s_{i}.
\end{array}
\right.
\end{equation}
where $s_{i}$ is the number of vanishing weights for each cusp or elliptic fixed point, $s_{i}=\dim\ker(I-\rho(S_{i}))$ ($i \leq m$), $s_{m+i}=\dim\ker(I-\rho(T_{i}))$ ($i >m$). The partial flags in the parabolic and elliptic data of $\overline{\rho}$ satisfy $[U'_{i}]=\left[\overline{U_{i}}\Pi_{i}\right]$, where $\Pi_{i}$ is any permutation matrix satisfying 
\begin{equation}
\Pi_{i}^{-1}W_{i}\Pi_{i}=\mathrm{diag}(\alpha_{i1},\dots,\alpha_{is_{i}},\alpha_{ir},\alpha_{ir-1}\dots,\alpha_{is_{i}+1}).
\end{equation}
\end{lemma}
\begin{proof}
The proof is clear, since ${}^{t}M^{-1}=\overline{M}$ whenever $M$ is unitary.
\end{proof}


\begin{proposition}\label{iso-reg-cusp}
Let $D_{e}=\sum e_{i}$, $D_{p}=\sum p_{j}$, and $E^{\vee}$ denote the dual bundle of a given bundle $E$. 
\begin{itemize}
\item[(i)] If $\rho$ has nonvanishing elliptic weights, then
\begin{equation}\label{iso-regular}
\mathfrak{M}_{2}(\Gamma,\rho)\cong H^{0}\left(S,E_{\rho}\otimes K_{S}(D_{p}+D_{e})\right),
\end{equation}
and moreover, when also all parabolic weights are different from 0, $\mathfrak{S}_{2}(\Gamma,\rho)\cong \mathfrak{M}_{2}(\Gamma,\rho)$.
\item[(ii)] There is an isomorphism
\begin{equation}\label{iso-cusp}
\mathfrak{S}_{2}(\Gamma,\rho)\cong H^{0}\left(S,E_{\overline{\rho}}^{\vee}\otimes K_{S}\right).
\end{equation} 
in particular, it follows that 
\begin{equation}\label{iso-cusp-Ad}
\mathfrak{S}_{2}(\Ga,\Ad\rho)\cong H^{0}\left(S, E^{\vee}_{\Ad\rho}\otimes K_{S}\right).
\end{equation}
\end{itemize}
\end{proposition}
\begin{proof}
(i) Let $\phi=f(\tau)d\tau$. By the definition of the bundle $E_{\rho}$, it is only needed to analyze the behavior of $\phi$
near the cusps and elliptic fixed points. Since 
\[
d\tau = \frac{1}{2\pi\sqrt{-1}}\frac{dq}{q},
\]
when $f(\tau)$ is regular, the $q$-series expansions \eqref{Fourier} imply that near each cusp, the local expressions $q^{-W_{i}}_{i}U_{i}^{-1}\phi$ are meromorphic $q_{i}$-differentials with potential simple poles at each cusp.
In turn, since $\rho$ has nonvanishing elliptic weights, the local expansions \eqref{elliptic} 
imply that near the elliptic fixed points, $$\zeta^{-W_{i}}_{i}U^{-1}_{i}\phi
=\left(\sum_{k=0}^{\infty}\mathbf{b}_{i}(k)\zeta_{i}^k\right)\frac{d\zeta_{i}}{\zeta_{i}}.$$  
According to lemma \ref{trans}, these local differentials determine a  section of the bundle $E_{\rho}\otimes K_{S}(D_{p}+D_{e})$ in the trivializations of the neighborhoods $\{\cU_{i}\}_{i=1}^{n+m}$ of the cusps and elliptic fixed points. This establishes \eqref{iso-regular}. Furthermore, the isomorphism between the spaces of regular and cusp forms follows from Remark \ref{vanishing-coef}. 

(ii) If $f$ is a cusp form, the expansions (\ref{Fourier}) satisfy the additional condition $(b_{i}(0))_{j}=0$ whenever $\alpha_{ij}=0$. Thus, near each cusp the local expression of the differential $\phi$ can be equivalently rewritten as 
\[
\sigma_{i}^{*}(\phi)=\overline{U'_{i}}q^{-W'_{i}}\left(\sum_{k=0}^{\infty}\mathbf{b}'_{i}(k)q^{k}\right)dq
\]
and similarly, 
near each elliptic fixed point $e_{i}$
\[
\phi=\overline{U'_{i}}\zeta^{-W'_{i}}_{i}\left(\sum_{k=0}^{\infty}\mathbf{b}'_{i}(k)\zeta^{k}_{i}\right)d\zeta_{i}
\]
The transition functions of the bundles $E_{\overline{\rho}}^{\vee}$, $E_{\overline{\rho}}$ are related as $g^{\vee}_{ij}={}^{t}g^{-1}_{ij}$, 
and it follows from lemmas \ref{trans} and \ref{dual} that the cusp form condition, together with the holomorphicity condition of $f$ at elliptic fixed points, is precisely the holomorphicity condition of sections of $E_{\overline{\rho}}^{\vee}\otimes K_{S}$ in the neighborhoods $\{\mathscr{U}_{i}\}_{i=1}^{n+m}$. Therefore, the isomorphism \eqref{iso-cusp} is established. 
Notice how, in the special case when all parabolic and elliptic weighs are nonzero, lemmas \ref{trans} and \ref{dual} imply the isomorphism
\[
E^{\vee}_{\bar{\rho}}\cong E_{\rho}(D_{e}+D_{p}).
\]
Let now $\rho$ be arbitrary unitary. 
The representation $\Ad\rho$ has the additional property of being self-dual, i.e., equivalent to its contragradient representation, with isomorphism determined by transposition. Consequently, there is an isomorphism 
\[
\mathfrak{S}_{2}(\Ga,\Ad\rho)\cong\mathfrak{S}_{2}\left(\Ga,\overline{\Ad\rho}\right),\qquad \Phi\mapsto {}^{t}\Phi. 
\] 
and also an induced bundle isomorphism $E_{\overline{\Ad\rho}}\cong E_{\Ad\rho}$. This concludes the proof.
\end{proof}
\begin{corollary}\label{dimension-reg-cusp}
\leavevmode
\begin{itemize}
\item[(i)] Let $\rho$ have nonvanishing elliptic weights. The dimension of the space of regular forms is
\[
\dim\mathfrak{M}_{2}(\Gamma,\rho)=\dim H^{0}\left(S,E_{\rho}^{\vee}(-D_{p}-D_{e})\right)+r(g+m+n-1)-\sum_{i,j}\alpha_{ij}.
\]
\item[(ii)] For arbitrary $\rho$, the dimension of the space of cusp forms is
\[
\dim\mathfrak{S}_{2}(\Gamma,\rho)=
\left(\dim\CC^{r}\right)^{\overline{\rho}}+r(g-1)
+\sum_{i,j}\alpha '_{ij},
\]
in particular,
\[
\dim\mathfrak{S}_{2}(\Gamma,\Ad\rho)=\dim\left(\End\CC^{r}\right)^{\Ad\rho}+r^{2}(g-1)+\sum_{i=1}^{n+m}\dim_{\CC}\mathcal{F}_{W_{i}}
\]
where $\mathcal{F}_{W_{i}}$ is the partial flag manifold associated to $W_{i}$ (see remark \ref{flag manifold}).
\end{itemize}
\end{corollary}
\begin{proof}
The first two statements are a direct consequence of Proposition \ref{iso-reg-cusp}, Remark \ref{degree-sections} and the Riemann-Roch theorem for vector bundles. For the last identity, we also note that the representation $\Ad\rho$ would inherit the parabolic weights 
\[
\theta_{i\{jk\}}=\left\{
\begin{array}{cc}
\alpha_{ij}-\alpha_{ik} & \text{if}\quad j\geq k,\\
\min\{0,1-\alpha_{ij}+\alpha_{ik}\} & \text{if}\quad j<k.
\end{array}\right.
\]
In particular, it is always true that for fixed $i$, $\displaystyle\sum_{j,k}\theta_{i\{jk\}}=\dim_{\CC}\mathcal{F}_{W_{i}}$.
\end{proof}
\begin{remark}
Hejhal \cite{Hejhal83} has provided a computation of the dimensions of the spaces of regular and cusp forms, 
in the same generality than the present work, with an approach based on the Selberg trace formula. 
\end{remark}

\section{Poincar\'e series of weight 2}\label{theo}

We will now provide a construction of the Poincar\'e series, relative to the \emph{a priori} choice of parabolic and elliptic generators for $\Ga$, and will determine an infinite spanning set for the space of cusp forms, parametrized by the coefficients of the series expansions \eqref{Fourier} and \eqref{elliptic}. For this, it is enough to construct weak solutions to the Cauchy--Riemann equation on a suitable Hilbert space by means of a limiting procedure.
For $p,q=0,1$, let $\cC^{p,q}_{c}(\Ga,\rho)$ be the space of smooth, $\CC^{r}$-valued functions on $\HH$ satisfying
\[
f(\gamma\tau)\gamma'(\tau)^{p}\overline{\gamma'(\tau)^{q}}=\rho(\gamma)f(\tau)\qquad \tau\in\HH,\quad\gamma\in\Ga,
\]
and which are compactly supported within a fundamental region $\mathfrak{F}$. For any 
$f_{1},f_{2}\in\cC^{p,q}_{c}(\Ga,\rho)$, the \emph{Petersson inner product} is defined as 
\begin{equation}\label{Petersson}
\langle f_{i},f_{2}\rangle_{P}=2^{p+q}\iint\limits_{\mathfrak{F}}\langle f_{1},f_{2}\rangle\;y^{2p+2q-2}dx dy
\end{equation}
where $\langle f_{1},f_{2}\rangle={}^{t}f_{1}\overline{f_{2}}$ (or $\tr(f_{1}f_{2}^{*})$ in the matrix-valued case)
and 
$\tau=x+\sqrt{-1}y$. We emphasize the difference between the norms $\parallel\cdot\parallel$ and $\parallel\cdot\parallel_{P}$.
A fundamental property of this inner product is its invariance under the action of $\Ga$. Let us denote by $\cH^{p,q}(\Gamma,\rho)$ the completion of $\cC^{p,q}_{c}(\Ga,\rho)$ to a Hilbert space under the Petersson inner product. As a consequence of the $q$-series expansions at the cusps, 
\[
\cH^{1,0}(\Gamma,\rho)\cap \mathfrak{M}_{2}(\Ga,\rho)=\mathfrak{S}_{2}(\Ga,\rho),
\]
so it is also possible to define, in terms of the Cauchy--Riemann operator $\bar{\partial}:\cH^{1,0}(\Gamma,\rho)\to\cH^{1,1}(\Gamma,\rho)$, the space of cusp forms of weight 2 as
\[
\mathfrak{S}_{2}(\Ga,\rho)=\ker\bar{\partial}.
\]
 
\begin{definition}
An \emph{admissible multi-index} is a triple $I=(i,j,l)$ defined by the conditions $1\leq i\leq n+m$, $1\leq j\leq r$ and $l\geq 0$, so that $l>0$ if $\alpha_{ij}=0$. Thus $I$ is admissible if and only if $\kappa_{I}=\alpha_{ij}+l>0$. We will make the distinction $1\leq i\leq m$ (elliptic), and $m + 1 \leq i \leq n + m$ (parabolic). 
\end{definition}
\noindent For convenience, we will use the notation
\[
(f|_{2p,2q}\gamma)(\tau)= \rho(\gamma)^{-1}f(\gamma\tau)\gamma'(\tau)^{p}\overline{\gamma'(\tau)^{q}},
\]
called \emph{slash operation} in the number theory jargon. Then an automorphic form of weight 2 is a function satisfying $f |_{2} \gamma=f$ $\forall\gamma\in\Ga$.


\subsection{elliptic series} Let $\Ga_{i}=\langle S_{i}\rangle$, $1 \leq i \leq m$, and let 
\[
\zeta(\tau)=(\tau-\sqrt{-1})/(\tau+\sqrt{-1})
\]
be the Cayley map. For any $s>0$, let
\[
\cF^{s}_{i,l}(\tau)= \left(U_{i}\zeta^{\nu_{i}(W_{i}+lI)+I}(1-|\zeta|^{2})^{s}\zeta'\right)\circ \sigma_{i}^{-1}
\]
and consider the Poincar\'e series\footnote{The elliptic Poincar\'e series will be constructed following Petersson's convention \cite{Petersson41}, together with the addition of a suitable convergence factor.}
\begin{equation}\label{Poincare2}
P_{I}^{s}(\tau)= \displaystyle\sum_{\gamma\in\Gamma}
\left(\cF^{s}_{i,l} |_{2} \gamma \right)
\mathbf{e}_{j}, 
\end{equation} 
 together with the auxiliary series 
\begin{equation}\label{delbar-Poincare2}
Q_{I}^{s}(\tau)= \displaystyle\sum_{\gamma\in\Gamma}
\left((\zeta\overline{\zeta'})\circ\sigma^{-1}_{i} \cF^{s-1}_{i,l} |_{2,2} \gamma \right)
\mathbf{e}_{j}. 
\end{equation}

\subsection{parabolic series} Let $\Ga_{i}=\langle T_{i - m}\rangle$, $m + 1\leq i \leq m + n$, and let $y=\im(\tau)$. For any $s>0$, let 
\[
\cF^{s}_{i,l}(\tau)=\left(U_{i}q^{W_{i}+lI} y^{s}\right)\circ \sigma_{i}^{-1}
\]
and consider the series 
\begin{equation}\label{Poincare1}
P^{s}_{I}(\tau)=\displaystyle
\sum_{\gamma\in\Gamma_{i}\setminus\Gamma}
\left(\cF^{s}_{i,l} |_{2} \gamma\right)
\mathbf{e}_{j},
\end{equation}
which are independent of the choice of representatives in $\Ga_{i}\setminus\Ga$, and satisfy (\ref{automorphic}) formally, together with the auxiliary series 
\begin{equation}\label{delbar-Poincare1}
Q^{s}_{I}(\tau)= \sum_{\gamma\in\Gamma_{i}\setminus\Gamma}
\left(\cF^{s-1}_{i,l} |_{2,2} \gamma\right)
\mathbf{e}_{j},
\end{equation}
which satisfies formally the functional equation
\[
Q^{s}_{I}(\gamma\tau)|\gamma'(\tau)|^{2}=\rho(\gamma)Q^{s}_{I}(\tau),\qquad \ga\in\Ga,\;\tau\in\HH.
\] 

\begin{lemma}\label{abs-conv}
For any admissible multi-index $I=(i,j,l)$ and $s>0$, the series $P^{s}_{I}(\tau)$ and $Q^{s}_{I}(\tau)$ are absolutely convergent for any $\tau\in\HH$ and uniformly convergent on compact sets. 
\end{lemma}  
\begin{proof}
This is essentially a consequence of the integral test for Poincar\'e series: for any function $\Theta:\HH \to \CC^{r}$ satisfying 
\[
\iint\limits_{\Gamma_{i}\setminus\HH}\|\Theta\|y^{p+q-2}dxdy <\infty
\]
the Poincar\'e series 
\[
\sum_{\gamma\in\Ga_{i}\setminus\Ga}\left(\Theta |_{2p,2q}\gamma\right)
\]
is absolutely convergent for any $\tau\in\HH$ and uniformly convergent on compact sets. In the parabolic case, the integral is performed over $[0,1)\times(0,\infty)$, and is equal to $\Ga(s)/(2\pi\kappa_{I})^{s}$ for both $P^{s}_{I}$ and $Q^{s}_{I}$. In the elliptic case, the integral is performed over the whole $\HH$ (which after a change of variables, yields an integral over the unit disk), and is equal to $2\pi\mathrm{B}((\nu_{i}\kappa_{I}+3)/2,s) =2\pi\Ga(s)\Ga((\nu_{i}\kappa_{I}+3)/2)/\Ga(s+(\nu_{i}\kappa_{I}+3)/2)$ for $P^{s}_{I}$, and $\pi\mathrm{B}(\nu_{i}\kappa_{I}/2+2,s)=\pi\Ga(s)\Ga((\nu_{i}\kappa_{I}/2+2)/\Ga(s+\nu_{i}\kappa_{I}/2+2)$ for $Q^{s}_{I}$.
\end{proof}

\begin{remark}
There is an alternative proof of lemma \ref{abs-conv}. First, for the parabolic Poincar\'e series, consider the series of norms of the summands in \eqref{Poincare1}.  Since for any $\eta\in\PSL$,  $\im(\eta(\tau))=|\eta'(\tau)|\im(\tau)$, we obtain 
\[
  \frac{1}{\im(\tau)}\cdot
\displaystyle\sum_{\gamma\in\Gamma_{i}\setminus\Gamma}
e^{-2\pi\kappa_{I}\im(\sigma_{i}^{-1}\circ\gamma)}
\im(\sigma_{i}^{-1}\circ\gamma)^{1+s}, 
\]
Similarly, for the series of norms of the summands of \eqref{delbar-Poincare1}, we get
\[
 \frac{1}{\im(\tau)^{2}}\cdot
\displaystyle\sum_{\gamma\in\Gamma_{i}\setminus\Gamma}
e^{-2\pi\kappa_{I}\im(\sigma_{i}^{-1}\circ\gamma)}
\im(\sigma_{i}^{-1}\circ\gamma)^{1+s}.
\]
Since $\kappa_{I}>0$, up to the first factor, all of these series are majorized by the Eisenstein series 
\[
E_{i}(\tau,1+s)=\displaystyle\sum_{\gamma\in\Gamma_{i}\setminus\Gamma}
\im(\sigma_{i}^{-1}\circ\gamma)^{1+s},
\] 
which is absolutely convergent for any $s>0$ and $\tau\in\HH$, uniformly convergent on compact sets, and moreover, has the asymptotic behavior at every cusp $p_{k}$
\[
E_{i}(\sigma_{k}\tau,s)-y^{1+s}=O(y^{-s})\qquad \text{as}\quad y\to\infty.
\]
\cite{Kubota73}.  In fact, since 
\[
y = \displaystyle\frac{1-|\zeta|^{2}}{|1-\zeta |^{2}},
\]
the series of absolute values constructed from the elliptic Poincar\'e series \eqref{Poincare2} and \eqref{delbar-Poincare2} are also majorized by $E_{i}(\tau, 1+s)$. In particular, we conclude that for any cusp $p_{k}$,
\begin{equation}\label{asymptotic}
\left\{\parallel\im(\sigma_{k}\tau)P_{I}^{s}(\sigma_{k}\tau)\parallel,\;
\parallel\im(\sigma_{k}\tau)^{2}Q_{I}^{s}(\sigma_{k}\tau)\parallel\right\}=O(y^{-s})
\end{equation}
as $\im(\tau)\to\infty$.
\end{remark}

\begin{lemma}\label{integrable}
For any admissible multi-index $I$ and $s>0$, $P_{I}^{s}(\tau)\in\mathscr{H}^{1,0}(\Gamma,\rho)$ and $Q_{I}^{s}(\tau)\in\mathscr{H}^{1,1}(\Gamma,\rho)$. The norms $\|P^{s}_{I}\|_{P}$ and $\|Q^{s}_{I}\|_{P}$ are smooth and bounded as functions of $s$. 
\end{lemma}
\begin{proof}
Consider the automorphic functions  $\parallel yP_{I}^{s}(\tau)\parallel$ and $\parallel y^{2}Q_{I}^{s}(\tau)\parallel$. 
It follows from the estimate \eqref{asymptotic} and lemma \ref{abs-conv} that for any cusp $p_{k}$  and $s>0$,
\begin{equation}
\displaystyle\lim_{y\to\infty}\parallel \im(\sigma_{k}\tau)P_{I}^{s}(\sigma_{k}\tau)\parallel=\lim_{y\to\infty}\parallel \im(\sigma_{k}\tau)^{2}Q_{I}^{s}(\sigma_{k}\tau)\parallel=0.
\end{equation}
In particular, $\parallel yP_{I}^{s}(\tau)\parallel$ and $\parallel y^{2}Q_{I}^{s}(\tau)\parallel$ are uniformly bounded in $\mathfrak{F}$. Since $\mathfrak{F}$ has finite hyperbolic area, this readily implies that 
\[
\parallel P_{I}^{s}(\tau)\parallel_{P}<\infty,\quad\parallel Q_{I}^{s}(\tau)\parallel_{P}<\infty,
\] 
and moreover, the functions $s\mapsto \|P^{s}_{I}\|_{P}$, $s\mapsto \|Q^{s}_{I}\|_{P}$ are bounded in any open interval $(0,\epsilon)$. Since $\lim_{s\to0} \|P^{s}_{I}\|_{P}$, $\lim_{s\to0} \|Q^{s}_{I}\|_{P}$ exist, we conclude that the Hilbert space limits $\lim_{s\to0} P^{s}_{I}$ and $\lim_{s\to0} Q^{s}_{I}$ also exist.

\end{proof}

\begin{definition}\label{def-Poincare}
Given an admissible multi-index $I=(i,j,l)$, and $\kappa_{I}=\alpha_{ij}+l$, the \emph{Poincar\' e series of weight 2} for the representation $\rho$ are defined as the Hilbert space limit 
\begin{equation}
P_{I}(\tau)=\left\{
\begin{array}{lr}
\displaystyle 2\pi\kappa_{I}\cdot\lim_{s\to 0}P^{s}_{I}(\tau), & \text{(parabolic)},\\\\
\displaystyle \frac{\nu_{i}\kappa_{I}+1}{2\pi}\cdot\lim_{s\to 0}P^{s}_{I}(\tau), & \text{(elliptic)}.
\end{array}\right.
\end{equation}
\end{definition}
\begin{corollary}
The functions $P_{I}(\tau)$ are holomorphic, that is, $P_{I}(\tau)\in\mathfrak{S}_{2}(\Ga,\rho)$.  
\end{corollary}
\begin{proof}
This is a consequence of lemmas \ref{abs-conv} and \ref{integrable}, since  the functions  $P^{s}_{I}(\tau)$ and $Q^{s}_{I}(\tau)$ were defined so that  for any $s>0$, the equation
\[
\bar{\partial} P^{s}_{I}(\tau)=
\left\{
\begin{array}{lr}
-sQ^{s}_{I}(\tau) & i \leq m\\\\
-\left(s/2\sqrt{-1}\right)Q^{s}_{I}(\tau) & i >m,
\end{array}\right.
\] 
is satisfied in the weak sense.
In particular, $\bar{\partial}P_{I}^{s}(\tau)\to 0$ weakly as $s\to 0$ in $\cH^{1,1}(\Ga,\rho)$, which implies that $\delb P_{I}(\tau)=0$ weakly. 
It follows from Weyl's lemma \cite{Kra72} that $P_{I}$ is indeed holomorphic, thus a cusp form.
\end{proof}
The final result in this section shows that our construction overdetermines $\mathfrak{S}_{2}(\Gamma,\rho)$, an essential property that is satisfied by the Poincar\'e series of weight $> 2$.
\begin{proposition}\label{inner-Poincare}
Let $f(\tau)$ be a cusp form and $I=(i,j,l)$ a fixed admissible multi-index. Then 
\begin{equation}\label{inner}
\langle f(\tau),P_{I}(\tau)\rangle_{P}=b_{i}(l)_{j}.
\end{equation}
\end{proposition}
\begin{proof}
Let us consider the parabolic Poincar\'e series first. The fundamental region $\mathfrak{F}$ can be chosen so that $\sigma_{i}^{-1}\cdot \mathfrak{F}$ lies in the semi-strip  $0\leq\re(\tau)\leq 1$ in $\mathbb{H}$, and the representatives of any given class $[\gamma]$ in $\Gamma_{i}\setminus\Gamma$ can be chosen so that $\sigma_{i}^{-1}\gamma\cdot \mathfrak{F}$ satisfies the same property, and moreover, 
\[
\bigsqcup_{\gamma\in\Ga_{i}\setminus\Ga} \left(\sigma_{i}^{-1}\gamma\cdot \mathfrak{F}\right)=[0,1)\times(0,\infty)\subset\HH. 
\]
This is the basis of the classical ``unfolding trick" for the expression of the inner product of the series \eqref{Poincare1} with any $f\in\cH^{1,0}(\Ga,\rho)$ (resp. $g\in\cH^{1,1}(\Ga,\rho)$) as the elementary integral\footnote{This manipulation is genuine as a consequence of the complete additivity of the Lebesgue integral.}
\[
\langle P_{I}^{s}(\tau),f(\tau)\rangle_{P}= 2 \displaystyle\int_{0}^{\infty}\int_{0}^{1}
\langle U_{i}q^{W_{i}}\mathbf{e}_{j},f(\sigma_{i}\tau)\sigma_{i}'(\tau)\rangle
q^l
y^{s}
d x d y,
\]
Consider the $q$-series expansions \eqref{Fourier} for $f(\tau)$ at the cusp $p_{i}$. The inner product $\langle P^{s}_{I}(\tau),f(\tau)\rangle_{P}$ can thus be estimated as
\[
\begin{array}{cl}

 & 2\displaystyle\sum_{k=0}^{\infty}\overline{b_{i}(k)_{j}} \left(\int_{0}^{1}e^{2\pi\sqrt{-1}(l-k)x} dx\right)\cdot\left(\int_{0}^{\infty}e^{-2\pi(k+l+2\alpha_{ij})y}y^{s} dy\right)\\\\
= & \displaystyle\frac{\Gamma(1+s)}{(2\pi\kappa_{I})^{1+s}}\cdot \overline{b_{i}(l)_{j}}.
\end{array}
\]

Now, let us consider the elliptic Poincar\'e series. Then, for the chosen fundamental region $\mathfrak{F}$ we have that 
\[
\bigsqcup_{\gamma\in\Ga} \left(\varphi_{i}^{-1}\gamma\cdot \mathfrak{F}\right)=\mathbb{H}. 
\]
This allows us to express Petersson's inner product $\langle P_{I}^{s}(\tau),f(\tau)\rangle_{P}$ as the elementary integral
\[
2\displaystyle\int_{0}^{2\pi}\int_{0}^{1}
\langle U_{n+i}\zeta^{\nu_{i}(W_{i}+lI)+I}(1-|\zeta|^{2})^{s}\mathbf{e}_{j},f(\sigma_{i}\tau)\sigma_{i}'(\tau)/\zeta'\rangle
r d r d \theta,
\]
where $\zeta=re^{i\theta}$. Consider the series expansions \eqref{elliptic} for $f(\tau)$ at the elliptic fixed point $e_{i}$. As with the parabolic Poincar\'e  series, the inner product $\langle P^{s}_{I}(\tau),f(\tau)\rangle_{P}$ can now be estimated as
\[
\begin{array}{cl}

 & 2\displaystyle\sum_{k=0}^{\infty}\overline{b_{i}(k)_{j}} \left(\int_{0}^{2\pi}e^{\sqrt{-1}(l-k)\theta} d\theta \right)\cdot\left(\int_{0}^{1}r^{\nu_{i}(2\alpha_{ij}+k+l)+1}(1-r^{2})^{s} dr\right)\\\\
= & \displaystyle 2\pi\frac{\Gamma(\nu_{i}\kappa_{I}+1)\Ga(s+1)}{\Ga(\nu_{i}\kappa_{I}+s+2)}\cdot \overline{b_{n+i}(l)_{j}}.
\end{array}
\]
Thus, recalling definition \ref{def-Poincare}, the claim of the proposition will follow if we let $s\to 0$.
\end{proof}

\begin{remark}\label{nonzero}
A special feature of the representations of the form $\Ad\rho$, is that the sheaf of local sections of the bundle $E_{\Ad\rho}$ corresponds, by definition, to the sheaf of endomorphisms of $E_{\rho}$ that preserve the parabolic structure, i.e., there is an isomorphism of vector bundles $E_{\Ad\rho} \cong \textrm{Par}\End E_{\rho}$. Thus, after considering Serre duality, proposition \ref{iso-reg-cusp} implies the isomorphism\\
\centerline{
\begin{xy}
(0,0)*+{\fS_{2}(\Ga,\Ad\rho)^{\vee}}="a";
(60,0)*+{H^{1}(S,\textrm{Par}\End E_{\rho})}="b"; 
{\ar@{->}^{\cong}_{\text{\v{C}ech-Serre}} "a";"b"}
\end{xy}}

The latter space parametrizes infinitesimal deformations of the parabolic structure in the bundle $E_{\rho}$. Thus, it follows from the natural isomorphism $H^{0}(S,E^{\vee})\cong H^{0}(S,E)^{\vee}$, together with Petersson inner product, that the correspondence $\Phi\mapsto \langle\cdot,\Phi\rangle_{P}$ determines every infinitesimal deformation of the parabolic bundle $E_{\rho}$. If in particular, we let $\Phi=P_{I}$, we obtain a geometric interpretation of the Fourier coefficients (as linear functionals in $\fS_{2}(\Ga,\Ad\rho)$) for the matrix-valued cusp forms. 

Now, every element of $\mathfrak{n}(W_{i})$, $i=1,\dots,m+n$, parametrizes a nontrivial infinitesimal deformation of the parabolic structure in $E_{\rho}$, and there is an inclusion $\mathfrak{n}(W_{i})\hookrightarrow\fS_{2}(\Ga,\Ad\rho)$, $i=1,\dots,m+n$. Thus, taking canonical bases, we conclude  from proposition \ref{inner-Poincare} that the associated Poincar\'e series, obtained after taking $l=0$ and corresponding to the entries of the matrices $B_{i}(0)$, are all different from 0.

\end{remark}

Let $\cI=\{(i,j,l): 1\leq i\leq n, 1\leq j\leq r, l\geq 0\}$  and $\{f_{1},\cdots f_{d}\}$ be an arbitrary basis for $\mathfrak{S}_{2}(\Ga,\rho)$. Consider the function 
\[
\cT:\cI\to\CC^{d},\qquad I\mapsto(w_{1I},\cdots, w_{dI})=\mathbf{w}_{I},\quad w_{kI}=\langle f_{k},P_{I}\rangle_{P} 
\]
defined up to the left action of $\GL(d,\CC)$.
In particular, Petersson's result on linear relations of cusp forms \cite{Petersson40} (cf. \cite{Lehner64}) is still valid in the present context. 
\begin{corollary}
Let $\{P_{I}(\tau)\}$ be a collection of Poincar\' e series of weight 2, and $\{\lambda_{I}\}$ a collection of complex numbers. 
The relation
\[
\sum\lambda_{I}P_{I}(\tau)=0
\]
is satisfied if and only if 
\[
\sum\overline{\lambda_{I}}\mathbf{w}_{I}=0.
\]
\end{corollary}
\begin{proof}
Straightforward  from $\langle f_{k},\sum\lambda_{I}P_{I}\rangle_{P}=\sum\overline{\lambda_{I}}w_{kI}$; the left-hand side can be equal to 0 for all $k$ if and only if $\sum\lambda_{I}P_{I}(\tau)=0$, while by definition, the right-hand side is equal to zero for all $k$ if and only if $\sum\overline{\lambda_{I}}\mathbf{w}_{I}=0$.
\end{proof}

\section{The genus zero case}\label{genus 0}
We are interested in constructing an explicit basis for the spaces $\mathfrak{S}_{2}(\Ga,\rho)$, $\mathfrak{S}_{2}(\Ga,\Ad\rho)$ when $g=0$, a condition that will be assumed henceforth throughout this section. 

\begin{remark}
To the best of our knowledge, the basis problem in arbitrary genus has not been solved for arbitrary $H$-groups even in the classical case of scalar cusp forms with a trivial character. For the corresponding genus 0 case, see \cite{Smart63}.
\end{remark}

The Birkhoff--Grothendieck theorem captures the fundamental difference between the cases $g=0$ and $g>0$, since for the former, it states that every vector bundle over $\CC\PP^{1}$ is holomorphically equivalent to a direct sum of line bundles of the form $\mathcal{O}(a)$,
\[
E \cong \bigoplus_{j=1}^{r}\mathcal{O}(a_{j}), \qquad a_{1}\leq\cdots\leq a_{r}.
\] 
In particular, the space $H^{0}(\CC\PP^{1},E)$ can be identified with the space of vector-valued polynomials of degree at most $a_{j}$ in the $j$th-entry (or 0 if $a_{j}<0$). Thus, 
\[
\dim H^{0}(\CC\PP^{1},E)=\sum_{j=1}^{r} 
\mathrm{max}\{a_{j}+1,0\}.
\]
Let 
\[
N = \begin{pmatrix}
a_{1} & & 0\\
& \ddots & \\
 0 & & a_{r}
\end{pmatrix}
\]
The Birkhoff--Grothendieck splitting $\bigoplus_{j=1}^{r}\mathcal{O}(a_{j})$ of a bundle $E$ is equivalent to the existence of a pair of trivializations of $E$ over the affine charts $\{\CC_{0},\CC_{1}\}$ of $\CC\PP^{1}$, with transition function $z^{N}:\CC^{*}\to\mathrm{GL}(r,\CC)$.

\begin{remark}\label{even-split}
Given an arbitrary set of weights $\mathcal{W} = \{W_{i}\}_{i=1}^{m+n}$, we say that a unitary representation $\rho$ of the group $\Gamma$ is \emph{admissible} if its associated weights are such.
The construction of any character variety $\cK$ of equivalence classes of irreducible unitary representations of the group $\Gamma$ is additionally dependent on an arbitrary choice of conjugacy classes for the elliptic and parabolic generators, which is, in turn, equivalent to a choice of weights $\mathcal{W}$, in such a way that every representative of every equivalence class in $\cK$ is an admissible representation for $\mathcal{W}$.
The collection of all irreducible admissible unitary representations forms a principal bundle $\cP\to\cK$, 
with structure group $\mathrm{PSU}(r)$, and right action corresponding to conjugation of representations. The Mehta-Seshadri theorem \cite{MS80} states that $\cK$ is analytically isomorphic to the moduli space $\cN$ of stable parabolic bundles over $\CC\PP^{1}$ of parabolic degree 0 and with given weights $\mathcal{W}$.    
Moreover, $\cN$ possesses a Zariski open set $\cN_{0}$ of \emph{evenly-split} bundles, satisfying the generic condition $|a_{j}-a_{k}|\leq 1$ $\forall\,j,k$. 
\end{remark}

Let us also consider a \emph{hauptmodul} $J:\HH^{*}\to\CC\PP^{1}$ (i.e. an automorphic function for $\Ga$ that is univalent on $\mathfrak{F}$ and regular at each cusp), normalized in such a way that $J(p_{n})=\infty$. We will denote the points $\{J(e_{i}), J(p_{i-m})\}$ by $z_{1},\dots,z_{m+n}$.
In the case when $E=E_{\rho}$ is constructed from a unitary representation of $\Ga$, the following result holds. 
(cf. \cite{MT14}). 

\begin{lemma}\label{unifor}
Let $\rho:\Ga\to\mathrm{U}(r)$ be an admissible representation for a fixed set of weights $\{W_{i}\}_{i=1}^{m+n}$. For every choice of representatives $\{U_{i}\}_{i=1}^{n+m}$ there exists a holomorphic function $Y:\HH\rightarrow\mathrm{GL}(r,\CC)$ satisfying 
\begin{equation} \label{auto-Y}
Y(\s\tau)=Y(\tau)\rho(\s)^{-1},\quad \forall\,\s \in\Ga,\;\tau\in\HH,
\end{equation}
with $\zeta$-series expansions at each elliptic fixed point
\begin{equation}\label{Y-elliptic}
Y(\sigma_{i}\tau)=\left(\sum_{k=0}^{\infty}C_{i}(k)\zeta^{\nu_{i}k}\right)\zeta^{-\nu_{i}W_{i}}U_{i}^{-1},\quad i \leq m,
\end{equation}
and having the $q$-series expansions
\begin{equation} \label{Y-Fourier}
Y(\sigma_{i}\tau)=\left(\sum_{k=0}^{\infty}C_{i}(k)q^{k}\right)q^{-W_{i}}U_{i}^{-1},\quad i>m,
\end{equation}
\begin{equation} \label{Y-Fourier-2}
Y(\sigma_{m+n}\tau)=q^{-N}\left(\sum_{k=0}^{\infty}C_{m+n}(k)q^{k}\right)q^{-W_{m+n}}U_{m+n}^{-1},
\end{equation}
where the constant terms $C_{i}(0)\in\mathrm{GL}(r,\CC)$ for $i=1,\dots,n+m$. 
The set $\Upsilon(\rho)$ of all functions $Y$ with these properties is a torsor for the group of  automorphisms $\mathrm{Aut}\,E_{\rho}$ of the bundle $E_{\rho}$. 
\end{lemma}

\begin{proof} It follows from the definition of the bundle $E_{\rho}$, in terms of the transition functions \eqref{transition-ext}, and the transition function of a Birkhoff--Grothedieck splitting, that an isomorphism 
\begin{equation}\label{eq:isomorphism}
E_{\rho} \cong \bigoplus_{j=1}^{r}\mathcal{O}(a_{j})
\end{equation}
is equivalent to the existence of a holomorphic map $Y:\HH \to \mathrm{GL}(r,\CC)$, satisfying \eqref{auto-Y}--\eqref{Y-Fourier-2}, and that if $\rho$ is fixed, any two such maps would differ by left multiplication of a map of the form $g\circ J |_{\HH}$, with $g:\CC_{0}\to\mathrm{GL}(r,\CC)$ holomorphic  and such that $z^{-N} g z^{N}|_{\CC^{*}}$ extends to a holomorphic map over $\CC_{1}$, which corresponds to postcomposition of the isomorphism \eqref{eq:isomorphism} by an automorphism of the Birkhoff--Grothendieck splitting.
\end{proof}

\begin{remark}
The function $Y:\HH\to\mathrm{GL}(r,\CC)$ plays the role of a vector bundle uniformization map, and establishes the isomorphism of parabolic bundles between  $E_{\rho}$, as an extension of the local system $\Ga\setminus\HH_{0}\times \CC^{r}$, and the splitting $\oplus_{j=1}^{r}\mathcal{O}(a_{j})$ together with a parabolic structure, modulo the action of its group of automorphisms.
More concretely, the unitary representation $\rho:\Gamma\to\mathrm{U}(r)$ determines a parabolic structure on the bundle $E_{\rho}$. A collection of decreasing flags\footnote{Decreasing flags are parametrized by the homogeneous space $\mathrm{B}(r)\setminus\mathrm{GL}(r,\mathbb{C})$, where $\mathrm{B}(r)$ is the Borel group of lower triangular matrices.} is induced in the following way: for each factorization 
\[
\rho(\gamma_{i})=U_{i}e^{2\pi\sqrt{-1}W_{i}}U_{i}^{-1},\qquad U_{i}=(u_{i,jk})
\]
we consider the column vectors
\[
\mathbf{v}_{ij}=(u_{i,1j},\cdots,u_{i,rj}),\qquad j=1,\cdots, r.
\]
and 
\[
F_{ij}=\mathrm{Span}(\mathbf{v}_{i,j},\mathbf{v}_{i,j+1},\cdots,\mathbf{v}_{i,r}),
\]
Then $\mathbb{C}^{r}=F_{i1}\supset F_{i2}\supset\cdots,\supset F_{ir}\supset\{\mathbf{0}\}$ are the flags over the fibers $\{e_{i}\}\times\mathbb{C}^{r}$, $\{p_{i-m}\}\times\mathbb{C}^{r}$.
Now, consider an arbitrary choice of bundle uniformization map $Y:\mathbb{H}\to\mathrm{GL}(r,\mathbb{C})$ and the bundle map $\mathscr{J}:E_{0}\hookrightarrow E_{\rho}$, where $E_{0}=\Ga\setminus\HH_{0}\times\CC^{r}$. Recalling that over $\mathbb{H}^{*}\times\mathbb{C}^{r}$, the local bases at the elliptic fixed points and cusps are given, respectively, by the vectors  $\{\zeta_{i}^{\alpha_{ij}}\mathbf{v}_{ij}\}_{j=1}^{r}$, $\{q_{i}^{\alpha_{ij}}\mathbf{v}_{ij}\}_{j=1}^{r}$, The map $\mathscr{J}$ induces decreasing flags over the fibers $\pi^{-1}([e_{i}]), \pi^{-1}([p_{i}])\subset E_{\rho}$ (once one considers the trivializations over the affine charts) by means of the correspondence
\[
\mathbf{w}_{ij}=\left. Yq_{i}^{\alpha_{ij}}\mathbf{v}_{ij}\right|_{\tau=\tau_{i}}
\] 
which are precisely the columns of the matrices $C_{i}(0)$, and the subspaces $$F'_{ij}=\mathrm{Span}(\mathbf{w}_{ij},\mathbf{w}_{ij+1},\cdots,\mathbf{w}_{ir}),$$
up to the action of the group $\Aut E_{\rho}$ in the bundle $E_{\rho}$. 

\end{remark}
Let $E_{\bar{\rho}}\cong\bigoplus_{j=1}^{r}\mathcal{O}(a_{j})$ be realized by a matrix-valued function $Y$ as in lemma \ref{unifor}. It follows from proposition \ref{iso-reg-cusp} that $$\mathfrak{S}_{2}(\Ga,\rho)\cong H^{0}\left(\CC\PP^{1},\bigoplus_{j=1}^{r}\mathcal{O}(b_{j})\right),$$ 
where $b_{j}=-a_{j}-2$. When $\rho$ is irreducible, it follows that $a_{j} < 0\quad \forall j$. Keeping this in mind, and since every unitary representation splits as a sum of irreducibles, we could assume that $\rho$ is either irreducible, or that at least no irreducible factor is trivial, since such trivial factor would lead to the standard scalar case.
\begin{theorem}\label{theo-1} 
Let $\Ga$ be a genus 0 Fuchsian group of the first kind. Let $\mathscr{J}$ be the collection of multi-indices $I=(i,j,0)$, with $|\mathscr{J}|=\dim\mathfrak{S}_{2}(\Ga,\rho)$, determined by lexicographic order, up to a permutation of the elliptic fixed points and cusps.\footnote{The restriction $i\neq n$ can be removed with the action of $\mathrm{PSL}(2,\CC)$ on the choice of $J$.}  
The Poincar\'e  series $\{P_{I}\}_{I\in\mathscr{J}}$  form a basis for $\mathfrak{S}_{2}(\Ga,\rho)$.
\end{theorem}
\begin{proof}
The fundamental idea behind the proof is to show that the number of zeros of any element of the orthogonal complement of the linear span of $\{P_{I}\}_{I\in\mathscr{J}}$ would have to be larger than allowed, and therefore such element would be trivial. 

Let us consider a hauptmodul $J$ as before. Then, each column $v_{k}$ of ${}^{t}YJ'$ 
corresponds to a meromorphic section of $\mathcal{O}(b_{k})$. We should distinguish the columns $\{v_{1},\dots,v_{t}\}$ ($t\leq r$) for which $b_{k}\geq 0$, since in this case, the functions
$(g_{k}\circ J)v_{k}$, where $g_{k}(z)$ is a polynomial of degree at most $b_{k}$, are cusp forms. In particular, the functions $\{v_{1},\dots,v_{t} \}$  are cusp forms, and for every fixed elliptic point $\{e_{i}\}$ or cusp $\{p_{i}\}$, $i\neq m+n$,  $\dim(\mathrm{Span}\{\mathbf{b}_{i1}(0),\dots,\mathbf{b}_{it}(0)\})=t$ (as a consequence of proposition \ref{iso-reg-cusp}, the univalence of $J$ on $\mathfrak{F}$ (cf. \cite{MT14}), and lemma \ref{unifor}, since $\det C_{i}(0)\neq 0$). Therefore at least $t$ of the $r$ Poincar\'e series $\{P_{I}\}$, $I=(i,j,0)$ would be nonzero. Since these Poincar\'e series are obviously linearly independent, 
\[
t\leq \dim\mathfrak{S}_{2}(\Ga,\rho).
\]
Similarly, 
the following rough upper bound follows from corollary \ref{dimension-reg-cusp} 
\begin{equation}\label{bound}
\dim\mathfrak{S}_{2}(\Ga,\rho)=\sum_{j=1}^{t}(b_{j}+1)< r(m+n).
\end{equation}
Notice that $m+n\geq 3$ is a necessary condition when $g=0$, and that indeed, $r(m+n)$, the number of multi-indices of the form $(i,j,0)$, is always greater than the dimension of  $\mathfrak{S}_{2}(\Ga,\rho)$.

Consider the lexicographic subcollection $\mathscr{J}$ of the set of indices of the form $(i,j,0)$ with $|\mathscr{J}|=\dim\mathfrak{S}_{2}(\Ga,\rho)$, and assume that 
$\mathfrak{V}=\text{Span}\left(\{P_{I}\}_{I\in\mathscr{J}}\right)$ is a proper subset of $\mathfrak{S}_{2}(\Ga,\rho)$, so that $\mathfrak{V}^{\perp}$ is nontrivial. 

The case $r=1$ is straightforward, since any nonzero 
$f\in\mathfrak{V}^{\perp}$ would correspond to a holomorphic section with $b+1$ zeros by proposition \ref{inner-Poincare}, but $c_{1}(\mathcal{O}(b))=b$, a	 contradiction.

As a consequence of the bundle splittings, every cusp form $f(\tau)$ can be uniquely expressed as a linear combination
\begin{equation}\label{splitting}
f(\tau)=\sum_{k=1}^{t}(g_{k}\circ J)v_{k}.
\end{equation}
Pick any nonzero $f\in\mathfrak{V}^{\perp}$, and consider its splitting \eqref{splitting}. Let us express $\dim\mathfrak{S}_{2}(\Ga,\rho)=ct+d$, $c,d\in\ZZ$, $0\leq d<t$. By hypothesis, $\langle P_{I}(\tau),f(\tau)\rangle_{P}=0$ for all $I=(i,j,0)$, $1\leq i\leq c$, $1\leq j\leq t$. This implies that for each $1\leq k\leq t$, the polynomial $g_{k}(z)$ would have a zero at each $z_{i}$, $i=1,\dots, c$, since otherwise  
$\mathbf{b}_{ik}(0)=\lambda\cdot {}^{t}C_{i}(0)_{k}$
(the $k$th-column of the matrix ${}^{t}C_{i}(0)$) in the elliptic and $q$-series expansions \eqref{elliptic} and \eqref{Fourier} of $(g_{k}\circ J)v_{k}$ at $p_{i}$, with $\lambda\neq 0$, and thus there would be a nontrivial linear relation among several columns of at least one constant matrix $C_{i}(0)$, $1\leq i\leq c$, a contradiction. 
Therefore, the total number of zeros of the polynomials $g_{k}$ 
would be bounded below by $ct$, but
\[
ct>\sum_{k=1}^{t}b_{k}=\sum_{k=1}^{t}c_{1}(\mathcal{O}(b_{k})),
\]
a contradiction. This concludes the proof.
\end{proof}


For every $i=1,\dots,m+n$, consider the standard basis $\{\mathbf{e}_{jk}\}_{\alpha_{ij}-\alpha_{ik}>0}$ of the Lie algebra $\mathfrak{n}(W_{i})$, so that we can write $B_{i}(0)=\sum_{\alpha_{ij}-\alpha_{ik}>0} b^{jk}\mathbf{e}_{jk}$ (see remark \ref{q-series}). Altogether, through their canonical bases, the collection of Lie algebras $\{\mathfrak{n}(W_{i})\}_{i=1}^{m+n}$ parametrize a collection of Poincar\'e series $\{P_{I}\}_{I\in\cI}$, where $\cI$ is the set of multi-indices $\{ i,jk,0\}$ such that $\alpha_{ij}-\alpha_{ik}>0$.
Each Poincar\'e series in the collection $\{P_{I}\}_{I\in\cI}$ is clearly different from 0 (see remark \ref{nonzero}), and 
\[
|\cI|=\sum_{i=1}^{m+n}\dim_{\CC}\cF_{W_{i}}.
\]
Since $\dim H^{0}(\CC\PP^{1},E_{\Ad\rho})=\dim\left(\End\CC^{r}\right)^{\Ad\rho}\leq r^{2}$, we conclude when $g=0$ that
\[
\dim\mathfrak{S}_{2}(\Ga,\Ad\rho)\leq \sum_{i=1}^{m+n}\dim_{\CC}\cF_{W_{i}}.
\]
As a corollary of theorem \ref{theo-1}, we obtain the following result. Its significance is geometric, giving a precise formulation of how a geometric structure, namely, the flags in a parabolic bundle and their infinitesimal deformations, determine completely the  spaces of matrix-valued cusp forms of weight 2 when $S=\CC\PP^{1}$.

\begin{theorem}\label{theo-2} Let $\Ga$ be a genus 0 Fuchsian group, and $\cI$ the set of multi-indices determined by the collection of canonical bases of the Lie algebras $\mathfrak{n}(W_{i})$, $i=1,\dots,m+n$. Any lexicographic subcollection $\cI'\subset\cI$ satisfying $|\cI'|= \dim\mathfrak{S}_{2}(\Ga,\Ad\rho)$ spans $\mathfrak{S}_{2}(\Ga,\Ad\rho)$.
\end{theorem}






\noindent \textbf{Acknowledgments.}
I am deeply grateful to Leon Takhtajan for introdu-cing me to the subject, taking part in multiple discussions, 
and making many useful remarks during the creation of the present work. I will also like to thank Irwin Kra, 
Alberto Verjovsky, Scott Wolpert, and the referee, for providing important suggestions.

\bibliographystyle{amsalpha}
\bibliography{Poincare}

\end{document}